\documentclass[12pt]{amsart}

\usepackage{comment}
\usepackage{color}
\usepackage{amsmath}
\usepackage{amsthm}
\usepackage{amssymb}
\usepackage{graphicx}
\usepackage{enumerate}
\usepackage{amsfonts}
\usepackage{mathrsfs}
\usepackage{parskip}
\usepackage{mathdots}
\usepackage{color}

\numberwithin{equation}{section}
\theoremstyle{plain}
\newtheorem{Proposition}[equation]{Proposition}
\newtheorem{Corollary}[equation]{Corollary}
\newtheorem*{Corollary*}{Corollary}
\newtheorem{Theorem}[equation]{Theorem}
\newtheorem*{Theorem*}{Theorem}
\newtheorem{Lemma}[equation]{Lemma}
\theoremstyle{definition}

\newtheorem{Remark}[equation]{Remark}

\allowdisplaybreaks

\usepackage{enumitem}
\setlist[enumerate,1]{label=(\alph*),font=\upshape}

\setlist[enumerate,2]{label=(\roman*),font=\upshape}

\def\HH{\mathscr{H}}
\def\MM{\mathscr{M}}
\def\Mult{\mathfrak{M}}

\def\NN{\mathcal{N}}
\def\h{\mathcal{H}}

\def\C{\mathbb{C}}

\def\D{\mathbb{D}}
\def\T{\mathbb{T}}

\def\phi{\varphi}

\renewcommand{\ker}{\operatorname{Ker}}

\newcommand{\beqa}{\begin{eqnarray*}}
\newcommand{\eeqa}{\end{eqnarray*}}

\renewcommand{\leq}{\leqslant}
\renewcommand{\le}{\leqslant}
\renewcommand{\subset}{\subseteq}

\title[Smirnov class]{The Smirnov class for de Branges--Rovnyak  spaces}

\author[Fricain]{Emmanuel Fricain}
 \address{Laboratoire Paul Painlev\'e, Universit\'e Lille 1, 59 655 Villeneuve d'Ascq C\'edex }
 \email{emmanuel.fricain@math.univ-lille1.fr}

\author[Hartmann]{Andreas Hartmann}
\address{Institut de Math\'ematiques de Bordeaux, Universit\'e Bordeaux 1, 351 cours de la Lib\'eration 33405 Talence C\'edex, France}
\email{Andreas.Hartmann@math.u-bordeaux1.fr}

\author[Ross]{William T. Ross}
	\address{Department of Mathematics and Computer Science, University of Richmond, Richmond, VA 23173, USA}
	\email{wross@richmond.edu}
	
		\author[D.~Timotin]{Dan Timotin}
	\address{Simion Stoilow Institute of Mathematics of the Romanian Academy, PO Box 1-764, Bucharest 014700, Romania}
	\email{Dan.Timotin@imar.ro}

\thanks{The first author was supported by the Labex CEMPI (ANR-11-LABX-0007-01) and the project FRONT (ANR-17-CE40-0021). }


\begin{document}

\begin{abstract}
Using an explicit construction, we show that the de Branges-Rovnyak spaces $\mathscr{H}(b)$, corresponding to rational $b$, are contained in their associated Smirnov classes.
\end{abstract}

\maketitle

\section{Introduction}

%This paper gives a structure theorem for the de Branges--Rovnyak spaces $\mathscr{H}(b)$ similar to that of the well-known Hardy space. To provide proper context, let us set our notation and review the situation for the Hardy space. 

Let $H^2$ denote the Hardy space of analytic functions $f$ on the open unit disk $\D$ for which 
$$\sup_{0 < r < 1} \int_{\T} |f(r \xi)|^2 dm(\xi)$$ is finite. In the above, $\T = \partial \D$ is the unit circle and $m$ is normalized Lebesgue measure on $\T$. One can see that $H^2$ contains $H^{\infty}$, the bounded analytic functions on $\D$, and that any $\phi \in H^{\infty}$ satisfies $\phi H^{2} \subset H^2$. In fact, these are the only analytic functions on $\D$ which multiply $H^2$ into itself. Furthermore, $\phi H^2$ is dense in $H^2$ if and only if $\phi$ is a bounded outer function.  A well-known theorem of Smirnov \cite[p.~16]{Duren} says that every function $f \in H^2$ can be written as 
$f = \phi/\psi,$
where $\phi, \psi \in H^{\infty}$ and $\psi$ is outer.

Taking inspiration from $H^2$, one can define, for a general reproducing Hilbert space $\h$ of analytic functions on $\D$, the {\em Smirnov class} of $\h$, denoted by $\NN^{+}(\h)$, to be the set of quotients 
$\phi/\psi$,
where $\phi$ and $\psi$ are multipliers of $\h$ ($\phi \h \subset \h$ and $\psi \h \subset \h$) and the denominator $\psi$ is a cyclic multiplier ($\psi \h$ is dense in $\h$). We say that $\mathcal{H}$ satisfies
the {\em Smirnov property} if $\mathcal{H}\subset \mathcal{N}^+(\mathcal{H})$.
 In general, multipliers of $\h$ belong to $H^{\infty}$ but not every $H^{\infty}$ function is a multiplier of $\h$. Furthermore, the density of $\psi \h$ in $\h$ is usually more complicated than the simple criterion that $\psi$ is a bounded outer function. Indeed, when the polynomials are dense in $\mathcal H$ and the shift operator $S_{\mathcal H}$ (the multiplication by $z$) acts as a bounded operator on $\mathcal H$, then $\psi \mathcal H$ is dense in $\mathcal H$ if and only if 
$\psi$ is a cyclic vector for $S_{\mathcal H}$. This problem of cyclicity is known to be notoriously difficult and is only completely understood for a handful number of cases. 

A remarkable result in \cite{MR3687947} says that reproducing kernel Hilbert  spaces $\h$ satisfying the complete Nevanlinna--Pick property, along with the normalizing condition $k_{z_0}(z)=1$ for all $z\in\mathbb D$ and some $z_0\in\mathbb D$, have the property that $\h \subset \NN^{+}(\h)$. Two related papers \cite{MR4096723, MR1911187} show that the harmonically weighted Dirichlet spaces $D(\mu)$ and a class of  de Branges--Rovnyak spaces $\HH(b)$ enjoy the complete Nevanlinna-Pick property  and the above mentioned normalizing condition and thus belong to their associated Smirnov classes. The purpose of this paper is to give an explicit construction which proves the Smirnov property for a large class of $\mathscr{H}(b)$ spaces including the results mentioned above.
% The purpose of this paper 
%is to give an explicit construction 
%to prove the Smirnov property in certain $\mathscr{H}(b)$ spaces 
%that  are not covered by the above  results.

 To state our main theorem (Theorem \ref{maintheorem}), we make the following definitions. Let
$$H^{\infty}_{1} = \Big\{\phi \in H^{\infty}: \|\phi\|_{\infty} :=  \sup_{z \in \D} |\phi(z)| \leq 1\Big\}$$
denote the closed unit ball of $H^{\infty}$ and for $b \in H^{\infty}_{1}$ the {\em de Branges--Rovnyak space} $\HH(b)$ is the reproducing kernel Hilbert space of analytic functions whose kernel is 
$$k^{b}_{\lambda}(z) = \frac{1 - b(z) \overline{b(\lambda)}}{1 - \overline{\lambda} z}, \quad \lambda, z \in \D.$$
See \cite{FM1, FM2, Sa} for the details concerning  $\mathscr{H}(b)$ spaces.
Recall that $f \in H^{\infty}$ is an {\em inner function} if the radial boundary function $f(e^{i \theta})$ satisfies $|f(e^{i \theta})| = 1$ almost everywhere with respect to normalized Lebesgue measure $m $ on the unit circle $\T$. An $f \in H^{\infty}$ is {\em outer} if 
$$f(z)= \exp\Big(\int_{\T} \frac{\xi + z}{\xi - z} \log |f(\xi)| dm(\xi)\Big), \quad z \in \D.$$ 
%$Every $f \in H^{\infty} \setminus \{0\}$ satisfies $\log |f| \in L^1(m)$. 
 The main result of our paper is the following. 

\begin{Theorem}\label{maintheorem}
If $b \in H^{\infty}_{1}$ is rational but not inner, then $\HH(b) \subset \NN^{+}(\HH(b))$.
%\begin{enumerate}
%\item If $b$ is not inner, then $\HH(b) \subset \NN^{+}(\HH(b))$.
%\item  If $b \in H^{\infty}_{1}$ is rational and outer and $r > 0$ then $\HH(b^r) \subset \NN^{+}(\HH(b^r))$.
%\end{enumerate}
\end{Theorem}

We discuss an extension of this result when $b$ is outer in Corollary \ref{outerrrrr}. Note that if $b$ is rational and inner, it must be a finite Blaschke product.  In this situation $\HH(b)$ becomes a model space $(b H^2)^{\perp}$ (in fact a finite dimensional one) and the multipliers of any model space are just the constant functions \cite{MR3720929}. Thus, Theorem \ref{maintheorem} is no longer true. The multipliers of $\HH(b)$ for general $b \in H_{1}^{\infty}$ were studied in \cite{MR1098860, FM2, MR1254125, MR1614726} but are far from being completely understood.

Isolated cases of Theorem \ref{maintheorem} can be gleaned from results in \cite{MR3687947, MR4096723, MR1911187}. 
More precisely, when $b$ is rational and not inner, there is a unique rational function $a$ such that $|a|^2+|b|^2=1$ almost everywhere on $\T$ and $a(0)>0$. First, one can prove that if the zeros of $a$ on $\mathbb T$ are all simple, then $\HH(b)$ coincides with a $\mathcal D(\mu)$ space (with equivalent norms) \cite{MR3110499} and so, by \cite{MR3687947} and \cite{MR1911187},  $\HH(b) \subset \NN^{+}(\HH(b))$. Second, if $b(0)=0$ and there exists an analytic function $h$ on $\mathbb D$ such that $h(b(z))=z$ for all $z\in\mathbb D$ and the function $(z-b(0))/h(z)$ extends to be analytic on $\mathbb D$ with
\[
\left|\frac{z-b(0)}{h(z)}\right|\leq |1-\overline{b(0)}z|, \quad \mbox{for $z \in\mathbb D$},
\] 
then  by \cite{MR4096723} $\mathscr{H}(b)$  has the compete Nevanlinna Pick property; since $k^b_0\equiv 1$, \cite{MR3687947} implies that $\HH(b)\subset \NN^+(\HH(b))$. In particular, the hypothesis of the above results exclude a large class of rational $b$.

One of the main reasons for this paper is to prove Theorem~\ref{maintheorem} directly for all rational $b$ (not just the cases mentioned in the previous paragraph), and  only using the tools of $\HH(b)$ spaces. In addition, our construction is explicit.
%and does not depend on an abstract factorization theorem of Leech as used in \cite{MR3687947}.

The main idea that makes all of this work is that when $b \in H_{1}^{\infty}$ is rational and not inner, then $\HH(b) = \MM(\overline{a})$ (with equivalent norms), where $\MM(\overline{a}) = T_{\overline{a}} H^2$ (equipped with the range norm), $a \in H^{\infty}_{1}$ satisfies $|a|^2 + |b|^2 = 1$ almost everywhere on $\T$, and $T_{\overline{a}}$ is the co-analytic Toeplitz operator on $H^2$. 
Furthermore, we know the exact structure of $\MM(\overline{a})$ (and hence $\HH(b)$), as well as its multiplier algebra.  Along the way to proving Theorem \ref{maintheorem}, we prove the following curious representation theorem for $H^2$ functions. 

\begin{Theorem}\label{Thm:hardy-factorization}
If $p$ is a polynomial whose roots all lie on $\T$, then any $h \in H^2$ can be written as 
\begin{equation}\label{eq:Thm:hardy-factorization}
h = \frac{u}{p v + 1},
\end{equation}
where $u, v \in H^{2}$, $pu, pv \in H^{\infty}$, and $p v + 1$ is outer.
\end{Theorem}

Our methods also yield a characterization of  the cyclic vectors for the shift operator $f \mapsto z f$ on $\MM(\overline{a})$, and hence $\mathscr{H}(b)$, for rational $b$ (Proposition \ref{lkjhgoiuyt}).

\section{The Smirnov class for general spaces}

Before getting to the Smirnov class for $\HH(b)$, let us make a few remarks about the general case where $\h$ is any reproducing kernel Hilbert space of analytic functions on $\D$. Define the {\em multiplier algebra} of $\h$ as
$$\Mult(\h) = \{ \phi \in \operatorname{Hol}(\D): \phi \h \subset \h\}.$$
Standard facts say that $\Mult(\h) \subset H^{\infty}$ and if the constant function $1$ belongs to $\h$, which will be the case of $\mathscr{H}(b)$ spaces we consider, then $\Mult(\h) \subset \h$ and hence  $\Mult(\h) \subset \h \cap H^{\infty}$. Moreover,  when $\phi \in  \Mult(\h)$, an application of the closed graph theorem, says that the multiplication operator $M_{\phi} f = \phi f$ is a bounded operator on $\h$. The quantity 
$$\|M_{\phi}\| = \sup_{f \in \h, \|f\| = 1} \|\phi f\|, $$ the operator norm of $M_{\phi}$, is called the {\em multiplier norm} of $\phi$ and satisfies $\|\phi\|_{\infty} \leq \|M_{\phi}\|$. For the Hardy space $H^2$ it is a standard fact that $\Mult(H^2) = H^{\infty}$ and $\|M_{\phi}\| = \|\phi\|_{\infty}$. For many other spaces, $\Mult(\h)$ is a proper subset of $H^{\infty}$ and  it may happen that $\|\phi\|_{\infty} < \|M_{\phi}\|$. 

The {\em Smirnov class} of $\h$ is
$$\NN^{+}(\h) := \Big\{\frac{\phi}{\psi}: \phi, \psi \in \Mult(\h), \overline{\phi \h} = \h\Big\}.$$
In the above $\overline{\phi \h}$ denotes the closure of $\phi \h$ in the norm of $\h$. 
%Using the standard operator theory fact that 
%$$\overline{\phi \h} = \h \iff \ker M_{\phi}^{*} = \{0\},$$ we see that 
%$$\NN^{+}(\h) := \Big\{\frac{\phi}{\psi}: \phi, \psi \in \Mult(\h),\ker M_{\phi}^{*} = \{0\}\}.$$
The following fact is standard but we recall it here for the reader's convenience. 

\begin{Lemma}\label{bbBBbv}
$\NN^{+}(\h)$ is an algebra.
\end{Lemma}

\begin{proof}
For a multiplier $\phi$ we have $\overline{\phi \h} = \h$ if and only if $\ker M_{\phi}^{*} = \{0\}$. 
For $f_1, f_2 \in \NN^{+}(\h)$, we need to prove that $f_1 + f_2$ and $f_1  f_2$ belong to $\NN^{+}(\h)$. Write 
$$f_{j} = \frac{\phi_j}{\psi_j}, \quad \phi_j, \psi_j \in \Mult(\h), \quad \ker M_{\psi_j}^{*} = \{0\}.$$
Since 
$$f_1 + f_2 = \frac{\phi_1 \psi_2 + \phi_2 \psi_1}{\psi_1 \psi_2}$$ and $\Mult(\h)$ is an algebra, it follows that $\phi_1 \psi_2 + \phi_2 \psi_1$ and $\psi_1 \psi_2 \in \Mult(\h)$. Furthermore, since $M_{\psi_1 \psi_2}^{*} = M_{\psi_2}^{*} M_{\psi_1}^{*}$ is injective, we see that $f_1 + f_2 \in \NN^{+}(\h)$. The proof that $f_1 f_2 \in \NN^{+}(\h)$ is done in a similar way. 
\end{proof}

While the following lemma might appear quite obvious, it is important to be stated here
since the complete Nevanlinna-Pick property, which plays quite some role in connection
with the Smirnov property, heavily depends on the norm of the space whereas the Smirnov
property does not.

\begin{Lemma}\label{poiuytrfdcvbhtr5678ujh}
Suppose $\h_1$ and $\h_2$ are reproducing kernel Hilbert spaces of analytic functions on $\D$ that are equal as sets. Then $\NN^{+}(\h_1) = \NN^{+}(\h_2)$. 
\end{Lemma}

\begin{proof}
The closed graph theorem shows that the norms on $\h_1$ and $\h_2$ are equivalent in that there are positive constants $c_1, c_2$ such that 
$$c_1 \|f\|_{\h_1} \leq \|f\|_{\h_2} \leq c_2 \|f\|_{\h_1}, \quad f \in \h_1 = \h_2.$$
Thus $\Mult(\h_1) = \Mult(\h_2)$, and, for a multiplier $\phi$, it follows  that $\phi \h_1$ is dense in $\h_1$ if and only if $\phi \h_2$ is dense in $\h_2$. Thus $\NN^{+}(\h_1) = \NN^{+}(\h_2)$.
\end{proof}

\section{Some reminders about $\HH(b)$ spaces}

Here are some of the basics of $\HH(b)$ spaces needed for this paper. The details can be found in \cite{FM1, FM2, Sa}. For $b \in H^{\infty}_{1}$, the de Branges--Rovnyak space $\HH(b)$ is the reproducing kernel Hilbert space induced by the positive definite kernel 
$$k^{b}_{\lambda}(z) = \frac{1 - b(z) \overline{b(\lambda)}}{1 - \overline{\lambda} z}, \quad \lambda, z \in \D.$$
More properties of the specific class of $\HH(b)$ spaces that we will consider are given in Proposition \ref{678uhnjk}.
%Since we will focus on a specific class of $\mathscr{H}(b)$ spaces
%we will not introduce the general definition of this
%space and its norm here but appeal to the form given
%in \ref{oiuyxdftcvgyvghK} below and an associated equivalent norm 
%in \ref{alternate8udheye}.  

When $b$ is not an extreme point of $H^{\infty}_{1}$, which is equivalent to the condition $\log(1 - |b|) \in L^1(\T)$ (which will always be the case in this paper since $b$ is rational but not inner), then there exists a unique (up to a unimodular scalar) outer $a\in H_1^\infty$ such that $|a|^2 + |b|^2 = 1$ almost everywhere on $\T$. The pair $(a, b)$ is call a {\em Pythagorean pair} and $a$ is known as the {\em Pythagorean mate} to $b$. Below are some facts from \cite{MR3503356} about Pythagorean pairs. 

\begin{Proposition}
Suppose $b \in H^{\infty}_{1}$ is rational but not inner. Then
\begin{enumerate}
\item $a$ is also rational; 
\item $(a, b)$ form a corona pair in that
$$\inf_{z \in \D}( |a(z)| + |b(z)| )> 0.$$
\end{enumerate}
\end{Proposition}

Two important classes of functions in $\HH(b)$, where $b$ is not an extreme point of $H^{\infty}_{1}$, are 
$$\MM(a) := a H^2 \quad \mbox{and} \quad \MM(\overline{a}) := T_{\overline{a}} H^2.$$
Here $T_{\overline{a}}$ is the co-analytic Toeplitz operator $T_{\overline{a}} f = P_{+}( \overline{a} f)$ on $H^2$. 
We mention that $a$ is assumed outer (guaranteeing  uniqueness up to a unimodular
scalar factor). 
%Then only the zeros of $a$ on $\T$ will contribute to the structure of $\MM(a)$ or $\MM(\bar{a})$ ($a$ cannot vanish inside $\D$
%and zeros outside the closure of $\D$ define invertible factors). 
We will denote the zeros of $a$ on $\T$ by $\xi_1, \ldots, \xi_n$ and
the corresponding multiplicity by $m_j$, and define 
\begin{equation}\label{nmnbhjhjk898}
a_1(z) = \prod_{j = 1}^{n} (z - \xi_j)^{m_j},
\end{equation}
Since $a$ cannot vanish on $\D$ and zeros outside the closure of $\D$ define invertible factors, we have 
$\MM(\overline{a}) = \MM(\overline{a_1})$ \cite[Prop.~2.7]{MR3967886}. 
The proposition below is from \cite{MR3110499, MR3503356}.

\begin{Proposition}\label{678uhnjk}
Suppose $b \in H^{\infty}_{1}$ is rational but not inner. Then
\begin{enumerate}
\item  $\MM(a_1) \subset \MM(\overline{a_1}) = \HH(b)$;
\item  If $\xi_1, \ldots, \xi_n$ are the zeros of $a_1$ on $\T$ with corresponding multiplicities $m_1, m_2, \ldots, m_n$, then 
\begin{equation}\label{oiuyxdftcvgyvghK}
\HH(b) = \MM(\overline{a_1}) = \Big(\prod_{j = 1}^{n} (z - \xi_j)^{m_j}\Big) H^2  \dotplus \mathscr{P}_{N - 1},
\end{equation}
where $N = m_1 + m_2 + \cdots + m_n$ and $\mathscr{P}_{N - 1}$ are the polynomials whose degree is at most $N - 1$, and $\dotplus$ signifies the algebraic direct sum. 
\item An equivalent norm on $\HH(b)$ is given by
\begin{equation}\label{alternate8udheye}
\|a_1 g + q\|^{2} = \|g\|_{H^2}^{2} + \|q\|_{H^2}^{2}.
\end{equation}
 where $f=a_1g+q$ is the decomposition given by \eqref{oiuyxdftcvgyvghK}.
\end{enumerate}
\end{Proposition}

%If $a$ is the Pythagorean mate for $b$ and
%\begin{equation}\label{nmnbhjhjk898}
%a_1(z) = \prod_{j = 1}^{n} (z - \xi_j)^{m_j},
%\end{equation}
%note that $\MM(\overline{a}) = \MM(\overline{a_1})$ \cite[Prop.~2.7]{MR3967886}. 
We will prove that $\HH(b)  \subset \NN^{+}(\HH(b))$ (Theorem \ref{maintheorem}) by using the fact that $\HH(b) = \MM(\overline{a_1})$ (as sets with equivalent norms) and then employing Lemma \ref{poiuytrfdcvbhtr5678ujh}. So we will now focus on proving the following result. 

\begin{Theorem}\label{bbvFFd}
If $a_1$ is a polynomial whose zeros are contained in $\T$, then $\MM(\overline{a_1}) \subset \NN^{+}(\MM(\overline{a_1}))$. 
\end{Theorem}

So far, we can readily identify the contents of $\MM(\overline{a_1})$. We can also readily identify the multipliers. Recall that an analytic function $\phi$ is a {\em multiplier} of $\MM(\overline{a_1})$ if $\phi \MM(\overline{a_1}) \subset \MM(\overline{a_1})$. The multiplier algebra for $\MM(\overline{a_1})$ is denoted by $\Mult(\MM(\overline{a_1}))$. As mentioned previously,  $\Mult(\MM(\overline{a_1})) \subset H^{\infty} \cap \MM(\overline{a_1})$ but equality does not always hold. The multiplier algebra for general $\HH(b)$ spaces has been studied in \cite{MR1098860, MR1254125, MR1614726} and is often a complicated class of functions that is related to the invertibility of the product of certain Hankel operators and their adjoints \cite{MR1254125}. Fortunately, in our case, they are particularly simple \cite{MR3967886}.

\begin{Proposition}\label{vghyujnbg77YY}
If $a_1$ is a polynomial of degree $N$ whose zeros lie in $\T$, then 
$$\Mult(\MM(\overline{a_1})) = \MM(\overline{a_1}) \cap H^{\infty}.$$ In particular, $\phi \in \Mult(\MM(\overline{a_1}))$ if and only if $\phi = a_1 \widetilde{\phi} + r$, where $\widetilde{\phi} \in H^2$, $r \in \mathscr{P}_{N - 1}$, and  $a_1 \widetilde{\phi} \in H^{\infty}$.
\end{Proposition}

The Smirnov class of $\MM(\overline{a_1})$ is then 
$$\NN^{+}(\MM(\overline{a_1}))  = \Big\{\frac{\phi}{\psi}: \phi, \psi \in \Mult(\MM(\overline{a_1})), \overline{\psi \MM(\overline{a_1})} = \MM(\overline{a_1})\Big\}.$$

Proposition \ref{vghyujnbg77YY} describes $ \Mult(\MM(\overline{a_1}))$. To describe $\mathcal{N}^{+}(\MM(\overline{a_1}))$, it remains to characterize the cyclic vectors  for the shift operator on 
$\MM(\overline{a_1})$. This is done in the following. We remind the reader that $a_1$ is a polynomial of the form in \eqref{nmnbhjhjk898}.

\begin{Proposition}\label{lkjhgoiuyt}
Let $a_1$ be a polynomial of degree $N$ whose zeros lie in $\T$ and 
let $\psi = a_1 \widetilde{\psi} + r \in \Mult(\MM(\overline{a_1}))$, where $\widetilde{\psi} \in H^{2}$, with $a_1 \widetilde{\psi} \in H^{\infty}$, and $r \in \mathscr{P}_{N - 1}$. Then the following are equivalent:
\begin{enumerate}
\item $\psi \MM(\overline{a_1})$ is dense in $\MM(\overline{a_1})$;
\item $\psi$ is outer and $\operatorname{gcd}(a_1, r) = 1$. 
\end{enumerate}
\end{Proposition}

%The proof of this needs a brief discussion about equivalent norms.  Since $a_1$ is outer, $T_{\overline{a_1}}$ is injective and the norm on $\MM(\overline{a_1})$ is the range norm $$\|T_{\overline{a_1}} f\|_{\overline{a_1}}= \|f\|_{H^2}.$$ Proposition \ref{678uhnjk} says that a typical element of $\MM(\overline{a_1})$ is $a_1 g + q$, where $g \in H^2$ and $q \in \mathscr{P}_{N - 1}$, and thus we can equivalently re-norm $\MM(\overline{a_1})$ by 
%\begin{equation}\label{alternate8udheye}
%\|a_1 g + q\|^{2} = \|g\|_{H^2}^{2} + \|q\|_{H^2}^{2}.
%\end{equation}
%Note the use of the facts that $\mathscr{P}_{N - 1}$ is finite dimensional and the sum in \eqref{oiuyxdftcvgyvghK} is an algebraic direct sum.

\begin{proof}
(b) $\Rightarrow$ (a): Assume $\psi$ is outer and $\operatorname{gcd}(a_1, r) = 1$. In view of \eqref{alternate8udheye}, we  use the following inner product on $\MM(\overline{a_1})$:
$$\langle a_1 h_1 + q_1, a_1 h_2 + q_2\rangle = \langle h_1, h_2\rangle_{H^2} + \langle q_1, q_2\rangle_{H^2}.$$
Let $f \in \MM(\overline{a_1})$ such that $f \perp  \psi \MM(\overline{a_1})$, that is, 
\begin{equation}\label{1029384746}
\langle f, \psi g\rangle = 0, \quad g \in \MM(\overline{a_1}).
\end{equation}
Writing $f = a_1 \widetilde{f} + p$ and applying \eqref{1029384746} to $g = a_1 h$ for every $h \in H^2$ we get 
$0 = \langle a_1 \widetilde{f} + p, a_1 \psi h\rangle =   \langle \widetilde{f}, \psi h\rangle_{H^2}.$
Since $\psi$ is outer, we have $\overline{\psi H^2} = H^2$ and hence $\widetilde{f} \equiv 0$. It remains to check that $p \equiv 0$. Since $\operatorname{gcd}(a_1, r) = 1$, apply the Euclidean algorithm to the ring of polynomials to obtain polynomials $q$ and $s$ such that 
$q r = a_1 s + 1.$
Now apply \eqref{1029384746} to $g = p q \in \MM(\overline{a_1})$ (recall that $\MM(\overline{a_1})$ contains the polynomials). Then 
$$0 = \langle p, \psi p q\rangle = \langle p, a_1 p q \widetilde{\psi} + p q r\rangle.$$  But note that
$$p q r = p (a_1 s + 1) = p a_1 s + p$$ and $p \perp a_1pq\widetilde\psi+a_1ps$. Hence,
$0 = \langle p, p \rangle_{H^2}$ which yields $p \equiv 0$. 

(a) $\Rightarrow $ (b): Assume that $\psi \MM(\overline{a_1})$ is dense in $\MM(\overline{a_1})$. 
Since $\MM(\overline{a_1})$ embeds
boundedly into $H^2$
and polynomials belong to
$\MM(\overline{a_1})$ we have
$$
\C[z] \subset \MM(\overline{a_1})
=\overline{\psi \MM(\overline{a_1})}^{\MM(\overline{a_1})}
\subset \overline{\psi \MM(\overline{a_1})}^{H^2}
\subset \overline{\psi H^2}^{H^2}
$$
and so $\overline{\psi H^2}
=H^2$.
By Beurling's theorem \cite{Duren}, $\psi$ is outer. Now we need to check that $\operatorname{gcd}(a_1, r) = 1$. It follows from \eqref{oiuyxdftcvgyvghK} (see \cite[Cor.~4.11]{MR3850543} for the details) that the radial limit of every function in $\MM(\overline{a_1})$ exists at each $\xi_j$ (the zeros of the polynomial $a_1$) and 
$$|f(\xi_j)| \leq C_j \|f\|, \quad f \in \MM(\overline{a_1}).$$ Using our assumption again that $\overline{\psi \MM(\overline{a_1})}  = \MM(\overline{a_1})$ we produce  a sequence $g_n \MM(\overline{a_1})$ such that $\|1 - \psi g_n\| \to 0$ as $n \to \infty$. If $\psi(\xi_j) = 0$ then 
$1 \leq C_j \|1 - \psi p_n\| \to 0$ which is an obvious contradiction. Thus $\psi(\xi_j) \not = 0$ for every $1 \leq j \leq n$. From the representation $\psi = a_1 \widetilde{\psi} + r$ we see that $r(\xi_j) \not = 0$ for all $j$. This means that the zeros of $r$ do not meet the zeros of $a_1$ and so $\operatorname{gcd}(a_1, r) = 1$. 
\end{proof}

\section{$\HH(b)$ is contained in its Smirnov class}

In this section, we give the proof of Theorem \ref{Thm:hardy-factorization} and Theorem \ref{maintheorem}.

\subsection*{Proof of Theorem \ref{Thm:hardy-factorization}}
For $h\in H^2$, set 
\begin{equation}\label{9976654}
 F(z)= 2 \, p(z)\int_{-\pi}^{\pi}\frac{e^{it}}{e^{it}-z}\,
 \frac{|p(e^{it})|}{p(e^{it})}|h(e^{it})|\frac{dt}{2 \pi}, \quad z \in \D.
\end{equation}
A partial fraction decomposition with respect to the variable $e^{it}$ yields
\begin{equation}\label{bhh}
 \frac{e^{it}p(z)}{(e^{it}-z)p(e^{it})}=\frac{z}{e^{it}-z}+
 \sum_{k=1}^n\sum_{l=1}^{m_k}\frac{p_{k,l}(z)}{(\xi_k-e^{it})^{l}},
\end{equation}
where $p_{k, l}(z)$ are polynomials.

% (see below for the justification of this).

%This is clear from the usual partial fraction decomposition. Indeed, set $x=e^{it}$
%and consider $a(z)$ as a coefficient independant of $x$:
%\[
 %\frac{xa(z)}{(x-z)a(x)}=\frac{B}{x-z}+
 %\sum_{k=1}^n\sum_{l=1}^{n_k}\frac{p_{k,l}}{(\zeta_k-x)^{l}},
%\]%
%where $B$ and $p_{k,l}$ do not depend on $x$ (but on $z$).
%Multiplying by $(x-z)$ and setting $x=z$ we see that $B=z$.
%Also, multiplying by $(x-z)$, and considering now $x$ as a coefficient, we can
%\\

%TBC : The coefficients $p$ are polynomials, and their degrees are less than $n_k-1$ or $N-1$,
%$N=\sum_{k=1}^nn_k$ ? This is not really needed in the proof.
%\\

Since $2z=e^{it}+z-(e^{it}-z)$, and with the above decomposition in mind, we get
\begin{align*}
F(z) &= \int_{-\pi}^{\pi}\frac{e^{it}+z}{e^{it}-z}\, |p(e^{it})h(e^{it})|\frac{dt}{2 \pi}
 -\int_{-\pi}^{\pi} |p(e^{it})h(e^{it})|\frac{dt}{2 \pi}\\ 
 & \quad + 2 \sum_{k=1}^n\sum_{l=1}^{m_k}{p_{k,l}(z)}\, \int_{-\pi}^{\pi}
 \frac{|p(e^{it})|}{(\xi_k-e^{it})^{l}} |h(e^{it})|\frac{dt}{2 \pi}.
\end{align*}

Note that $|p(e^{it})|/|\xi_k-e^{it}|^{l}$ is bounded, and so all the integrals appearing
in the above double sum are well defined coefficients. As a result we can write
\[
 F(z)=\underbrace{\int_{-\pi}^{\pi}\frac{e^{it}+z}{e^{it}-z} 
  |p(e^{it})h(e^{it})|\frac{dt}{2 \pi}}_{F_0(z)}+q(z),
\]
where $q$ is a polynomial. 
Hence
$
 F=F_0+q.
$
Since
$$\Re\Big(\frac{e^{i t} + z}{e^{i t} - z}\Big)= \frac{1 - |z|^2}{|e^{i t} - z|^2} > 0,$$
$F_0$ has positive real part.
We also let $Q=1-q$ so that
\[
 F+Q=F_0+q+Q=F_0+1.
\]
Setting
\[
 \alpha=\frac{Qh}{F_0+1}, \quad \beta=\frac{F}{F_0+1},
\]
we get
\[
 \frac{\alpha}{1-\beta}=\frac{Qh/(F_0+1)}{1-F/(F_0+1)}=\frac{qh}{F_0+1-F}=\frac{Qh}{Q}=h.
\] 
It remains to show that 
$$\frac{\alpha}{1 - \beta}$$ is of the form \eqref{eq:Thm:hardy-factorization}.

Since $\Re F_0>0$ it follows that  $|F_0+1|>1$. Since $Q$ is a polynomial and $h\in H^2$,
we have $\alpha\in H^2$.  For the boundedness of $p \alpha$ it suffices to check that $p h/(F_0+1)$ is bounded.
From the preceding discussions it is clear that $\Re F_0$ is the Poisson extension
of $|ph|$, and so
%Assume $|au|\ge 2(1+|\alpha|)$, then
\[
 \frac{|ph|}{|F_0+1|}\le \frac{|ph|}{\Re F_0+1}=\frac{|ph|}{|ph|+1}\le 1.
%\frac{|au|}{|F+1-\alpha z|}\le \frac{|au|}{|F|-(1+|\alpha|)}
% \le \frac{|au|}{\Re F -(1 +|\alpha|)}=\frac{|au|}{|au|}
\]
The function $\beta$ is bounded by the following argument
\[
 |\beta|=\frac{|F|}{|F_0+1|}=\frac{|F_0+ p|}{|F_0+1|}
\le \frac{|F_0|}{|F_0+1|}+\frac{|p|}{|F_0+1|},
\]
and since $\Re F_0>0$ both terms are bounded. From \eqref{9976654} the function $F$ is $p$ times the Cauchy transform of an $L^2$ function and thus $F = p g$, where $g \in H^2$. Hence 
$$\beta = \frac{p g}{F_0 + 1} \in p H^2.$$
Next we need to verify that 
\[
 1-\beta=1-F/(F_0+1)=\frac{F_0+1-F}{F_0+1}=\frac{Q}{F_0+1}
\]
is outer. % and that $GCD(a,p)=1$. 
This is equivalent to saying that $q$ does not vanish on $\D$.
Since $Q=1-q$, and the coefficients in $p$ are homogeneous in $h$, we can apply a 
rescaling argument on $h$ to make $p$ small and hence $Q$ is zero-free in $\D$.
%Observe that 
%\[
% 1-\beta=1-\frac{F}{F_0+1}=1-\frac{ag}{F_0+1}=1-a\underbrace{\frac{g}{F_0}}_{\in H^2}.
%\]
Thus we have checked that $h = \alpha/(1 - \beta)$ satisfies the properties of \eqref{eq:Thm:hardy-factorization}.

\subsection*{Proof of Theorem \ref{maintheorem}}
Observe that we have reduced this to proving Theorem \ref{bbvFFd}.

Since 
$\mathscr{P}_{N - 1} \subset \MM(\overline{a_1}) \cap H^{\infty} = \Mult(\MM(\overline{a_1}))$ and since $1 \cdot\MM(\overline{a_1})$ is trivially dense in $\MM(\overline{a_1})$, then $p   = p/1 \in \NN^{+}(\MM(\overline{a_1}))$ and the polynomial part of $\MM(\overline{a_1})$ can be written as required.. Thus by Lemma \ref{bbBBbv} we just need to show that 
$$a_1 H^2 \subset \NN^{+}(\MM(\overline{a_1})).$$
By Proposition \ref{lkjhgoiuyt} this is equivalent to showing that $a_1 H^2$ is a subset of the class of functions $$\frac{a_1 f_1 + r_1}{a_1 f_2 + r_2},$$
where  $f_1, f_2 \in H^{2}$ with $a_1 f_1, a_1 f_1 \in H^{\infty}$, $ r_1, r_2 \in \mathscr{P}_{N - 1}$, $a_1 f_2  + r_2$ is outer, and $\operatorname{gcd}(a_1, r_2) = 1$. But this follows from Theorem \ref{Thm:hardy-factorization}. Note that we can choose $r_1=0$ and $r_2=1$.

%Using the fact that 
%$$|f(r \xi)| = o(\frac{1}{\sqrt{1 - r}}), \quad r \to 1^{-}$$ for each $\xi \in \T$, we see that 
%$$0 = \lim_{r \to 1^{-}} \frac{a(r \xi_j) f_1(\xi_j) + r_1(r \xi_j)}{a(r \xi_j) f_2(r \xi_j) + r_2(r \xi_j)} = \frac{r_1(\xi_j)}{r_2(\xi_j)}, \quad 1 \leq j \leq n.$$
%Thus $r_1(\xi_j) = 0$ for all $1 \leq j \leq n$. One can also argue that 
%$$\lim_{r \to 1^{-}} \frac{r_1(r \xi_j)}{(r \xi_j - \xi_j)^{s}} = 0, \quad 0 \leq s \leq m_j - 1,$$ and thus, since $r_1 \in \mathscr{P}_{N - 1}$, we see that $r_1 \equiv 0$. Thus $a H^2 \subset \NN^{+}(\MM(\overline{a}))$ if and only if every $h \in H^2$ can be written as 
%\begin{equation}\label{kkhvVVGvG}
%h = \frac{u}{a v + r},
%\end{equation}
%where $u, v \in H^2$ with $a u, av \in H^{\infty}$, $a v + r$ is outer, and $\operatorname{gcd}(a, r) = 1$. So it remains to apply Theorem~\ref{Thm:hardy-factorization}.

We end with the following generalization of Theorem \ref{maintheorem}.

\begin{Corollary}\label{outerrrrr}
 If $b \in H^{\infty}_{1}$ is rational and outer and $r > 0$ then $\HH(b^r) \subset \NN^{+}(\HH(b^r))$.
\end{Corollary}

\begin{proof}
Note that $\HH(b^r) = \HH(b)$ \cite[Thm.~1.6]{MR3503356} and now use Theorem \ref{maintheorem} along with Lemma \ref{poiuytrfdcvbhtr5678ujh}.
\end{proof}

\begin{Remark}
Note that Theorem \ref{bbvFFd} is  true for $\mathcal M(\bar \psi)$, where $\psi=a_1\varphi$ and $\varphi$ is any $H^\infty$ outer function that is invertible in $H^\infty$.
\end{Remark}

\bibliographystyle{plain}

\bibliography{references}

\end{document}